\documentclass[a4paper,reqno]{amsart}
\usepackage{adjustbox,amsmath,amsthm,amssymb,mathrsfs,mathtools,microtype,tikz-cd,xcolor}
\usepackage[T1]{fontenc}
\usepackage[colorlinks=true,allcolors=blue!60!black]{hyperref}
\usepackage[capitalise]{cleveref}

\usepackage[left=3.5cm,right=3.5cm,top=3.25cm,bottom=3.25cm]{geometry}
\usepackage{setspace}
\theoremstyle{plain}
\newtheorem{thm}{Theorem}[section]
\newtheorem{prop}[thm]{Proposition} 
\newtheorem{lem}[thm]{Lemma}

\theoremstyle{definition} 
\newtheorem{defn}[thm]{Definition} 
\newtheorem{exmp}[thm]{Example}

\theoremstyle{remark}
\newtheorem{rem}[thm]{Remark}

\newcommand{\id}{\mathrm{id}}
\DeclareMathOperator{\ev}{ev}
\DeclareMathOperator{\ad}{ad}
\DeclareMathOperator{\Hom}{Hom}
\newcommand{\lie}[1]{\mathfrak{#1}}
\DeclareMathOperator{\MC}{MC}

\newcommand{\Bigwedge}{\mathord{\adjustbox{valign=B,totalheight=.55\baselineskip}{$\bigwedge$}}{}}
\makeatletter
\newcommand{\invlim}{\mathop{\mathpalette\varlim@{\leftarrowfill@\scriptstyle}}\nmlimits@}
\newcommand{\dirlim}{\mathop{\mathpalette\varlim@{\rightarrowfill@\scriptstyle}}\nmlimits@}
\makeatother

\title{Gauge equivalence for complete $L_{\infty}$-algebras}
\author{Ai Guan}
\address{Department of Mathematics and Statistics\\
Lancaster University\\
Lancaster\\ LA1 4YF\\ UK}
\email{a.guan@lancaster.ac.uk}

\begin{document}

\begin{abstract}
  We introduce a notion of left homotopy for Maurer--Cartan elements in $L_{\infty}$-algebras and $A_{\infty}$-algebras, and show that it corresponds to gauge equivalence in the differential graded case.  
  From this we deduce a short formula for gauge equivalence, and provide an entirely homotopical proof to Schlessinger--Stasheff's theorem. 
  As an application, we answer a question of T.~Voronov, proving a non-abelian Poincar\'e lemma for differential forms taking values in an $L_{\infty}$-algebra.
\end{abstract}

\maketitle

\section{Introduction}

A Maurer--Cartan element in a differential graded Lie algebra (dgla) $V$ is a degree 1 element $\xi \in V$ satisfying $d\xi + \frac{1}{2}[\xi,\xi]=0$.
It is important to understand homotopies between Maurer--Cartan elements; for example, deformation problems are governed by Maurer--Cartan elements up to an appropriate notion of homotopy~\cite{ss,man99}.
Thus many different notions of homotopy have been studied for Maurer--Cartan elements in dglas and, more generally, in $L_{\infty}$-algebras; see~\cite{dp16} for an up-to-date and extensive survey. 

The Schlessinger--Stasheff theorem~\cite{ss} states that two Maurer--Cartan elements in a pronilpotent dgla are Sullivan homotopic (called Quillen homotopic in~\cite{dp16}) if and only if they are gauge equivalent.
One goal of this paper is to provide an entirely homotopical proof of this result, and to extend to it to $L_{\infty}$-algebras and $A_{\infty}$-algebras under certain completeness conditions.
To do this, we introduce a new homotopy relation for Maurer--Cartan elements in complete $L_{\infty}$-algebras and $A_{\infty}$-algebras. 
Maurer--Cartan elements are interpreted as morphisms of commutative differential graded algebras (cdgas); this is reviewed in Section~\ref{prelim} along with other relevant background on $L_{\infty}$-algebras and $A_{\infty}$-algebras. 
Two Maurer--Cartan elements are then defined to be left homotopic if they are left homotopic between morphisms in the category of cdgas, equipped with the model category structure of~\cite{hin97}.

As motivation for our definition, in Section~\ref{gauge_left_dgla} we show that there is a model structure on the category of complete dglas, namely that of~\cite{lm15}, in which gauge equivalence coincides with left homotopy, a result also obtained by \cite{rob18}. 
The results in this section should be considered Koszul dual to the approach taken in the rest of the paper, where we choose to work in the setting of cdgas in order for results to immediately generalise to the setting of $L_{\infty}$-algebras.
There are also close parallels between the approach in Section~\ref{gauge_left_dgla} and the recent papers \cite{bm13,bfmt18}, in which it is shown that gauge equivalence coincides with left homotopy for a larger class of dglas with a different model structure. However, their result only holds in the generality of dglas, and the method used does not seem to easily generalise to $L_{\infty}$-algebras. 

The rest of the paper is organised as follows.
In Section~\ref{left_cdga}, we define left homotopy of Maurer--Cartan elements and use it to deduce a new short formula for gauge equivalence in \cref{result:gauge_short_a} and \cref{result:gauge_short_b}.
In Section~\ref{gauge_formulae}, we prove combinatorial formulae for left homotopy in terms of rooted trees and give a direct proof that left homotopy coincides with gauge equivalence. 
As an application, in Section~\ref{application}, we answer a question posed by Voronov in~\cite{vor12}, and prove a version of the Poincar\'e lemma for differential forms taking values in an $L_{\infty}$-algebra.

\subsection{Notation and conventions}

Throughout the paper, $k$ denotes a field of characteristic zero. All unadorned tensor products will be over $k$, and all vector spaces will be over $k$ and cohomologically $\mathbb{Z}$-graded.
The suspension $\Sigma V$ of a graded vector space $V$ is graded by $(\Sigma V)^i = V^{i+1}$.

We will often refer to pseudocompact vector spaces -- these are projective limits of finite-dimensional vector spaces, equipped with the inverse limit topology. 
In particular, the $k$-linear dual $V^*$ of a discrete vector space $V$ is pseudocompact. 
The dual of a pseudocompact vector space is defined to be its topological dual; hence $V \cong V^{**}$ is always true, and the natural map $x \mapsto ev_x$ is $\ev_x(v^*) = (-1)^{|v^*||x|}v^*(x)$ by the Koszul sign rule. 
For two pseudocompact vector spaces $V$ and $W$, the space of morphisms $\Hom(V,W)$ is assumed to mean the space of continuous linear maps, and the tensor product $V \otimes W$ is assumed to mean the completed tensor product.
If $V = \invlim_i V_i$ is pseudocompact and $W$ is discrete, then their tensor product is defined to be $V \otimes W = \invlim_i V_i \otimes W$; note that in general this is neither discrete nor pseudocompact.

We abbreviate a differential graded (Lie) algebra or unital commutative differential graded algebra by dg(l)a or cdga.
The completed tensor and symmetric algebras are respectively $\widehat{T}V
\coloneqq \prod_{i=0}^{\infty} T^i(V)$ and $\widehat{S}V \coloneqq \prod_{i=0}^{\infty} S^i(V) $.
We will be mostly working in the context of complete $L_{\infty}$-algebras (as defined in \cref{prelim}), so the completion is usually unnecessary.

\subsection*{Acknowledgements}
The author thanks Chris Braun and Andrey Lazarev for numerous helpful discussions during the writing of this paper, and for their comments and corrections on earlier drafts.
Thanks are also given to an anonymous referee for their careful reading and thoughtful comments and corrections.

\section{Preliminaries on $L_{\infty}$- and $A_{\infty}$-algebras}
\label{prelim}

We start by recalling basic facts on $L_{\infty}$- and $A_{\infty}$-algebras, Maurer--Cartan elements, Sullivan homotopy, and gauge equivalence for the differential graded case.
All definitions given in this section are standard and agree with those commonly found in the literature, except for the notion of completeness: in particular, the definition given here agrees with \cite{lm15} but not with \cite{bfmt18}.

\begin{defn}[following~\cite{hl09}]
  \label{def:inftystr}
  Let $V$ be a graded vector space. 
  \begin{enumerate}
    \item An \emph{$L_{\infty}$-structure} on $V$ is a continuous degree 1 derivation $m$ of the complete cdga $\widehat{S}\Sigma^{-1}V^*$, such that $m^2=0$ and $m$ has no constant term. 
      The pair $(V,m)$ is called an \emph{$L_{\infty}$-algebra}, and $(\widehat{S}\Sigma^{-1}V^*,m)$ is called its \emph{representing complete cdga}. 

      Given two $L_{\infty}$-algebras $(U,m_U)$ and $(V,m_V)$, an \emph{$L_{\infty}$-morphism} $U \to V$ is a continuous cdga map 
      $(\widehat{S}\Sigma^{-1}V^*,m_V) \to (\widehat{S}\Sigma^{-1}U^*,m_U)$.

    \item An \emph{$A_{\infty}$-structure} on $V$ is a continuous degree 1 derivation $m$ of the complete dga $\widehat{T}\Sigma^{-1}V^*$, such that $m^2=0$ and $m$ has no constant term.    
      The pair $(V,m)$ is called an \emph{$A_{\infty}$-algebra}, and $(\widehat{T}\Sigma^{-1}V^*,m)$ is called its \emph{representing complete dga}.

      Given two $A_{\infty}$-algebras $(U,m_U)$ and $(V,m_V)$, an \emph{$A_{\infty}$-morphism} $U \to V$ is a continuous dga map 
      $(\widehat{T}\Sigma^{-1}V^*,m_V) \to (\widehat{T}\Sigma^{-1}U^*,m_U)$.
  \end{enumerate}
\end{defn}

One recovers the standard definition of an $L_{\infty}$-structure as a sequence of graded maps as follows: 
By definition, the derivation $m$ is determined by its components $m_i \colon \Sigma^{-1}V^* \to S^i \Sigma^{-1}V^*$, $i \ge 1$.
Dualise the components $m_i$ and apply the canonical identification of $S_n$-invariants and $S_n$-coinvariants, to get graded symmetric maps $\ell_i \colon S^i\Sigma V \to \Sigma V$ of degree 1, with $m_i = \frac{1}{i!} \ell_i^*$.
The condition $m^2=0$ then translates into higher Jacobi identities.

Under the identification $S^i\Sigma V \cong \Bigwedge^i V$, an $L_{\infty}$-structure on $V$ is equivalently a sequence of graded antisymmetric brackets 
$[-,\dots,-]_i \colon \Bigwedge^i V \to V$ of degree $2-i$.
For later convenience, we adopt the convention that the graded symmetric and graded antisymmetric operations are related by $\ell_i = \Sigma\,[-,\dots,-]_i\,(\Sigma^{-1})^{\otimes i}$, so that by the Koszul sign rule, 
\[
  \ell_i(x_1,\dots\!,x_i) 
  = (-1)^{\sum_{j=1}^{i-1} (i-j)|x_j|} \,\Sigma\,[\Sigma^{-1}x_1,\dots,\Sigma^{-1}x_i].
\]

Analogously, an $A_{\infty}$-structure is equivalent to a sequence of graded maps $T^i V \to V$, $i \ge 1$, of degree $2-i$, satisfying higher associativity identities.
Note that factorials do not appear in the $A_{\infty}$-algebra case, because there is no need to identify invariants and coinvariants.

For our purposes, $V$ will often be a pseudocompact vector space instead of discrete. 
In this case, an $L_{\infty}$-structure, $A_{\infty}$-structure, etc., on $V$ is defined by replacing the complete cdga $\widehat{S}\Sigma^{-1}V^*$ in \cref{def:inftystr} with the cdga $S\Sigma^{-1}V^*$, and replacing the complete dga $\widehat{T}\Sigma^{-1}V^*$ with the dga $T\Sigma^{-1}V^*$.

Throughout this paper we will use the language of Quillen's closed model categories~\cite{qui69}; see~\cite{ds95} for a review. 
We will consider the categories of cdgas and dgas with the model category structures of~\cite{hin97}, in which the weak equivalences are quasi-isomorphisms and the fibrations are degreewise surjections. 
All objects are fibrant, and the cofibrant objects are described as below. 

\begin{defn}
  A \emph{Sullivan cdga} (resp.~\emph{Sullivan dga}) is defined to be a cdga of the form $SV$ (resp.~dga of the form $TV$) such that $V$ is a graded vector space admitting a filtration
  \[
    0 \subseteq V_0 \subseteq V_1 \subseteq V_2 \subseteq \dots \subseteq V,\quad V=\bigcup_{i \ge 0} V_i,
  \]
  that is compatible with the differential $d$, i.e.~$d(V_i) \subseteq SV_{i-1}$ (resp.~$d(V_i) \subseteq TV_{i-1}$) for all $i$.
\end{defn}

The cofibrant objects in the model categories of (c)dgas are precisely retracts of Sullivan (c)dgas. A proof appears in, for example,~\cite[Theorem 9.1]{pos11} for dgas; the same proof also works for cdgas.

\begin{defn} 
  An $L_{\infty}$-algebra (resp.~$A_{\infty}$-algebra)\/ $V$ is \emph{complete} if $V$ is pseudocompact and its representing cdga (resp.~dga) is cofibrant in the model category of cdgas (resp.~dgas).
\end{defn}

\begin{defn}\label{mcelement}\hfill
  \begin{enumerate}
    \item Let $(V,m)$ be a complete $L_{\infty}$-algebra and $A$ be a cdga.
      An element $\xi \in V \otimes A$ is \emph{Maurer--Cartan}\index{Maurer--Cartan!element, equation ($L_{\infty}$)} if it has degree 1 and satisfies the \emph{Maurer--Cartan equation}
      \[
        (\id \otimes d_A)(\xi) + \sum_{i \ge 1} \frac{1}{i!} [\xi,\dots,\xi]^A_i = 0,
      \]
      where $[-,\dots,-]_i^A$ is the $A$-linear extension of $[-,\dots,-]_i$.

    \item Let $(V,m)$ be a complete $A_{\infty}$-algebra and $A$ be a cdga.
      An element $\xi \in V \otimes A$ is \emph{Maurer--Cartan}\index{Maurer--Cartan!element, equation ($A_{\infty}$)} if it has degree 1 and satisfies the \emph{Maurer--Cartan equation}
      \[
        (\id \otimes d_A)(\xi) + \sum_{i \ge 1} m_i^A(\xi,\dots,\xi) = 0,
      \]
      where $m_i^A$ is the $A$-linear extension of $m_i$.
  \end{enumerate} 
  The set of all Maurer--Cartan elements in $V \otimes A$ is denoted $\MC(V,A)$.
  In the case where $A=k$, we write $\MC(V,k)$ simply as $\MC(V)$.
\end{defn}

\begin{rem}
  The completeness condition on the $L_{\infty}$- and $A_{\infty}$-algebra ensures that the infinite sums converge in \cref{mcelement}.
\end{rem}

Given any cdga $A$ and complete $L_{\infty}$-algebra $V$, a Maurer--Cartan element in the $L_{\infty}$-algebra $V \otimes A$ is represented by a cdga map $S\Sigma^{-1}V^* \to A$. 
Similarly, given any dga $A$, a Maurer--Cartan element in the $A_{\infty}$-algebra $V \otimes A$ is represented by a dga map $T\Sigma^{-1}V^* \to A$.
For details see, for example, Proposition 2.2 and Remark 2.3 in~\cite{cl11}. 

Next we consider a notion of homotopy for Maurer--Cartan elements.
Let $V$ be an $L_{\infty}$-algebra or $A_{\infty}$-algebra.
Consider the $L_{\infty}$-algebra or $A_{\infty}$-algebra $V[t,dt] \coloneqq V \otimes k[t,dt]$, where $k[t,dt]$ denotes the free cdga generated by a degree 0 symbol $t$ and a degree 1 symbol $dt$. 
Then there are two natural cdga maps $f_0, f_1 \colon k[t,dt] \to k$, defined by sending $t$ to 0 and 1 respectively, and sending $dt$ to 0.

\begin{defn}
  Let $(V,m)$ be a complete $L_{\infty}$-algebra or $A_{\infty}$-algebra, and $A$ be a cdga.
  Two elements $\xi, \eta \in \MC(V,A)$ are \emph{Sullivan homotopic}\index{Sullivan homotopic} if there exists an element $h \in \MC(V, A[t,dt])$ such that 
  $(\id \otimes f_0)(h) = \xi$ and 
  $(\id \otimes f_1)(h)  = \eta$.
\end{defn}

We recall the notions of a good path object and a right homotopy in the sense of Quillen's model categories. 
For an object $X$ in a model category $\mathbf{C}$, a \emph{path object} for $X$ is an object $X^I$ in $\mathbf{C}$ with a factorization
\begin{center}
  \begin{tikzcd}[sep=scriptsize]
    X \arrow[r,"i","\sim"']\arrow[rr,"\Delta_X",swap,bend right] & X^I \arrow[r,"p"] & \phantom{X^I} & \hspace{-4em} X \times X,
  \end{tikzcd}
\end{center}
with $i$ a weak equivalence, where $\Delta_X$ is the diagonal map $(\id_X,\id_X)$. 
The path object $X^I$ is called \emph{good} if additionally, $p$ is a fibration. 

The following result is well-known; see for example~\cite{laz13}.

\begin{prop}
  Let\/ $(V,m)$ be a complete $L_{\infty}$-algebra or $A_{\infty}$-algebra, and let $A$ be a cdga.
  Two Maurer--Cartan elements\/ $\xi, \eta \in \MC(V, A)$ are Sullivan homotopic if and only if their representing \textup(c\textup)dga maps are right homotopic in the model category of \textup(c\textup)dgas.
\end{prop}

\begin{proof}
  This is immediate from regarding $h \in \MC(V \otimes A[t,dt])$ as a cdga map $S\Sigma^{-1}V^* \to A[t,dt]$ or a dga map $T\Sigma^{-1}V^* \to A[t,dt]$. Then
  \begin{center}
    \begin{tikzcd}
      A \arrow[r,"i"] & A[t,dt] \arrow[r,shift left,"f_0"] \arrow[r,shift right,swap,"f_1"] & A
    \end{tikzcd}
  \end{center}
  is a good path object for $A$ in the model category of (c)dgas; here $i$ denotes the natural inclusion.
\end{proof}

Now consider the case where $V$ is a complete dgla, with differential $d$ and bracket $[-,-]$.
In this case, $(V \otimes A)^0$ is a Lie algebra and the \emph{gauge group} $G$ of $V \otimes A$ is defined by exponentiating $(V \otimes A)^0$. 
That is, $G$ consists of formal symbols $\{e^x : x \in (V \otimes A)^0\}$, with multiplication $e^x e^y = e^{x * y}$ given by the Baker--Campbell--Hausdorff (BCH) formula.

\begin{defn}
  \label{def:gauge}
  The \emph{gauge action} of $G$ on $\MC(V,A)$ is defined by
  \begin{equation}
    \label{eq:gauge_dgla}
    e^x \cdot \xi =  \xi + \sum_{n=1}^{\infty} \frac{(\ad_x)^{n-1}}{n!} (\ad_x\xi - dx).
  \end{equation} 
  Two Maurer--Cartan elements $\xi, \eta \in V \otimes A$ are said to be \emph{gauge equivalent} if they lie in the same orbit of the gauge action. 
  We write $\mathscr{MC}(V,A)$ for the quotient of $\MC(V,A)$ by the gauge action.

  If $V$ is a complete dga, we say that two Maurer--Cartan elements in $V \otimes A$ are gauge equivalent if they are \emph{gauge equivalent} in the corresponding dgla, taken with the commutator bracket.
\end{defn}

\begin{rem}
  Completeness of $V$ implies that $V$ is pronilpotent; thus the infinite series in the BCH formula and the above gauge action~\eqref{eq:gauge_dgla} converge.
  Indeed, the ascending filtration on $S\Sigma^{-1}V^*$ corresponds to a descending filtration on $V$, and the Sullivan condition on the differential of $S\Sigma^{-1}V^*$ corresponds to pronilpotence of $V^*$ with respect to this filtration.
\end{rem}

More generally, for a (not necessarily complete) dga $V$, the following definition of gauge equivalence can be found in, for example, \cite{chl}.

\begin{defn}
  Let $V$ be a dga. The \emph{gauge action} of the invertible elements of $V^0$ on $\MC(V,A)$ is defined by
  \begin{equation}
    \label{eq:gauge_dga}
    a \cdot \xi =  a \xi a^{-1} - da \cdot a^{-1}.
  \end{equation}
\end{defn}

This definition is equivalent to~\cref{def:gauge} in the case where $V$ is complete.

Our aim is to establish gauge equivalence as a \emph{left} homotopy in the model category of cdgas.
In particular, left and right homotopy coincide when the domain is cofibrant and the codomain is fibrant, so this would prove that Sullivan homotopy and gauge equivalence coincide in the case of dglas. 
We will see how this interpretation extends to Maurer--Cartan elements in $L_{\infty}$-algebras.

\section{Gauge equivalence as a left homotopy of DGLAs}
\label{gauge_left_dgla}

It this section, we show that gauge equivalence for complete dglas coincides with left homotopy between complete dgla morphisms, with respect to the model category structure of~\cite{lm15}. 
An analogous result is proved in \cite{bm13,bfmt18} for a different model structure.
The key to this construction lies in an alternative characterization of Maurer--Cartan elements in $V$, when $V$ is a dgla, as follows.
Let $L(x)$ be the free complete dgla generated by one element $x$ of degree 1, with differential $dx = -\frac{1}{2}[x,x]$.
Then there is a correspondence between the set $\MC(V)$ and the set of dgla morphisms $L(x) \to V$.
Thus, to describe a left homotopy in a model category of complete dglas we require a cylinder object for $L(x)$; such a cylinder is given by the \emph{Lawrence--Sullivan interval} introduced in~\cite{ls14}.

The Lawrence--Sullivan interval $\mathfrak{L}$ is the free complete dgla on three generators $a$, $b$, $z$, where $|a|=|b|=1$, $|z|=0$, with differential 
\begin{gather*}
  da+\frac{1}{2}[a,a]=0, \quad db+\frac{1}{2}[b,b]=0,\\
  dz = \ad_z(b)+\frac{\ad_z}{e^{\ad_z}-\id}(b-a).
\end{gather*} 
That is, the differential $d$ is defined such that $a$ and $b$ are Maurer--Cartan elements in $\lie{L}$, and are gauge equivalent by $a = e^z \cdot b$. 

\begin{prop}\label{result:cyl_dgla}
  Let $i_0, i_1 \colon L(x) \to \lie{L}$ be the natural inclusions and $p \colon \lie{L} \to L(x)$ be the natural projection, that is, $i_0(x)=a$, $i_1(x)=b$ and $p(a)=p(b)=x$, $p(z)=0$. Then 
  \begin{center}
    \begin{tikzcd}
      L(x) \arrow[r,shift left,"i_0"] \arrow[r,shift right,swap,"i_1"] & \lie{L} \arrow[r,"p"] & L(x)
    \end{tikzcd}
  \end{center}
  is a good cylinder object for $L(x)$ in the category of complete dglas, equipped with a model structure in which a morphism $f \colon (V,d) \to (V',d')$ is
  \begin{enumerate}
    \item a weak equivalence if\/ $S\Sigma^{-1}(V')^* \to S\Sigma^{-1}V^*$ is a weak equivalence in the category of cdgas.
    \item a fibration if it is surjective.
  \end{enumerate}
\end{prop}

\begin{proof}
  See \cite{bfmt18} Corollary 5.3 and Theorem 7.6.
\end{proof}

It is then straightforward to show that gauge equivalence corresponds to the notion of a left homotopy in the category of complete dglas.

\begin{prop}
  Let $V$ be a complete dgla. 
  Two Maurer--Cartan elements $\xi, \eta \in V$ are gauge equivalent if and only if there exists a dgla morphism $h \colon \lie{L} \to V$ such that $h(a)=\xi$ and $h(b) = \eta$. 
\end{prop}

\begin{proof}
  Consider the element $h(z) \in V$. 
  If $h$ is a dgla morphism, then $h(z)$ has degree 0, and $d(h(z)) = h(dz)$, from which a direct computation shows $\xi = e^{h(z)} \cdot \eta$.
  Conversely, if $\xi$ and $\eta$ are gauge equivalent by $\xi = e^x \cdot \eta$, then define $h(a) = \xi$, $h(b) = \eta$ and $h(z) = x$. 
  The same computation shows that $h$ is a dgla morphism.
\end{proof}

It is natural to ask how this result can be generalized to a notion of homotopy when $V$ is an $L_{\infty}$-algebra.
In~\cite{bm13a}, two Maurer--Cartan elements $\xi, \eta \in V$ are called \emph{cylinder homotopic} (terminology following~\cite{dp16}) if there exists an $L_\infty$-morphism $\mathfrak{L} \to V$ such that $h(a)=\xi$ and $h(b) = \eta$.
Then by~\cite[Proposition 4.5]{bm13a}, two Maurer--Cartan elements are cylinder homotopic if and only if they are Sullivan homotopic.
The resulting generalization, however, is no longer a left homotopy of complete dglas.

\section{Left homotopy of Maurer--Cartan elements}
\label{left_cdga}

A different approach will be taken in this section: loosely, we will work in the Koszul dual picture, and consider Maurer--Cartan elements in $L_{\infty}$-algebras and $A_{\infty}$-algebras by their representing (c)dga maps. 
We define the following notion of homotopy for Maurer--Cartan elements.

\begin{defn} \hfill
  \begin{enumerate}
    \item Let $V$ be a complete $L_{\infty}$-algebra and $A$ be a cdga. 
      Two Maurer--Cartan elements $\xi, \eta \in V \otimes A$ are \emph{left homotopic} if their representing cdga maps $S\Sigma^{-1}V^* \to A$ are left homotopic in the model category of cdgas.
    \item Let $V$ be a complete $A_{\infty}$-algebra and $A$ be a cdga. 
      Two Maurer--Cartan elements $\xi, \eta \in V \otimes A$ are \emph{left homotopic} if their representing dga maps $T\Sigma^{-1}V^* \to A$ are left homotopic in the model category of dgas.
  \end{enumerate}
\end{defn}

\subsection{The cylinder object for (c)dgas}
\label{cyl_cdga}

We recall a cylinder object for cdgas constructed in~\cite[Section 2.2]{fot08}.
Given a cofibrant cdga of the form $(SV, d)$, its \emph{cylinder} $C(SV)$ is defined to be the cdga $(S(V \oplus \bar{V} \oplus \widehat{V}), D)$, where $\bar{V} \cong \Sigma V$ and $\widehat{V} \cong V$, and the differential $D$ is defined by
\[
  D(v) = dv,\ D(\bar{v}) = \widehat{v},\ D(\widehat{v}) = 0.
\]
We also define a degree $-1$ derivation $s$ on $C(SV)$ by 
\[
  s(v) = \bar{v},\ s(\bar{v}) = s(\widehat{v}) = 0.
\]
Then $\theta \coloneqq [s,D] = sD + Ds$ is a derivation of degree 0, so $e^{\theta} = \sum_{n=0}^{\infty} \theta^n / n!$ is an automorphism of $C(SV)$. 
Explicitly, 
\[
  \theta(v) = sdv + \widehat{v},\ \theta(\bar{v}) = \theta(\widehat{v}) = 0,
\]
and inductively, $\theta^n (v) = (sD)^n (v)$ for $n \ge 2$ as $s^2 = 0$. Since $SV$ is a Sullivan cdga, $\theta^N (v) = 0$ for some $N$, and hence we have a convergent series 
\begin{equation}
  \label{eq:exptheta}
  e^{\theta} (v) = v + \widehat{v} + \sum_{n=1}^{\infty} \frac{(sD)^n(v)}{n!},\ 
  e^{\theta} (\bar{v}) = \bar{v},\ 
  e^{\theta} (\widehat{v}) = \widehat{v}.
\end{equation}

Analogously, given a cofibrant dga of the form $(TV, d)$, its \emph{cylinder} $C(TV)$ is defined to be the dga $(T(V \oplus \bar{V} \oplus \widehat{V}), D)$, with $D$ and $e^{\theta}$ defined as above.
This is a different cylinder to the one constructed by~\cite{bl77}.

\begin{prop}\label{result:cyl_cdga}
  Let $(SV,d)$ be a cofibrant cdga. Let $i \colon SV \to C(SV)$ be the natural inclusion and $p \colon C(SV) \to SV$ be the natural projection, that is, $i(v)=v$ and $p(v)=v$, $p(\bar{v})=p(\widehat{v})=0$. Then 
  \begin{center}
    \begin{tikzcd}
      SV \arrow[r,shift left,"i"] \arrow[r,shift right,swap,"e^{\theta} \circ\, i"] & C(SV) \arrow[r,"p"] & SV
    \end{tikzcd}
  \end{center}
  is a good cylinder object for $(SV,d)$ in the model category of cdgas.
  Analogously, if $(TV,d)$ is a cofibrant dga, then
  \begin{center}
    \begin{tikzcd}
      TV \arrow[r,shift left,"i"] \arrow[r,shift right,swap,"e^{\theta} \circ\, i"] & C(TV) = (T(V \oplus \bar{V} \oplus \widehat{V}), D) \arrow[r,"p"] & TV
    \end{tikzcd}
  \end{center}
  is a good cylinder object for $(TV,d)$ in the model category of dgas.
\end{prop}

Since $C(SV)$ and $C(TV)$ are good cylinder objects, $\xi$ and $\eta$ are left homotopic if and only if there exists a cdga morphism $H \colon C(SV) \to k$ or dga morphism $H \colon C(TV) \to k$ that is a left homotopy between their representing (c)dga maps.

\subsection{Left homotopy in (c)dgas}

From now on, let $V$ be a complete $L_{\infty}$-algebra; everything we say will have an obvious analogue for complete $A_{\infty}$-algebras.
Consider the vector space $V \oplus \bar{V} \oplus \widehat{V}$, where $\bar{V} \cong \Sigma^{-1} V$ and $\widehat{V} \cong V$. 
This is a complete $L_{\infty}$-algebra with differential
\[
  d(\xi) = d\xi,\ d(\bar{\xi}) = 0,\ d(\widehat{\xi}\,) = \bar{\xi},
\]
and all brackets defined as 0 on the second and third components. 
Then the representing cdga of $V \oplus \bar{V} \oplus \widehat{V}$ is isomorphic to $C(SU)$, where $U = \Sigma^{-1}V^*$, $\bar{U} = \Sigma^{-1}\bar{V}^* \cong \Sigma U$ and $\widehat{U} \cong U$, with differential $D$ as in \cref{cyl_cdga}.

Note that for any cdga $A$, an element in $(V \oplus \bar{V} \oplus \widehat{V}) \otimes A$ is Maurer--Cartan if and only if it is of the form $\xi + x + 0$ for some $\xi \in \MC(V \otimes A)$ and $x \in (V \otimes A)^0$.
By abuse of notation, we denote also by $\xi$ its representing cdga map $(SU,d) \to A$, and by $x$ its equivalent degree 0 linear map $\bar{U} \to A$.
Hence $x$ and $\xi$ together determine a cdga map 
$H_{\xi,x} \colon C(SU) \to A$
that is a left homotopy between $\xi$ and $x * \xi \coloneqq H_{\xi,x} \circ e^{\theta} \circ i$, by
\[
  H_{\xi,x}(u) = \xi(u),\  H_{\xi,x}(\bar{u}) = x(\bar{u}),\ H_{\xi,x}(\widehat{u}) = 0. 
\]

Our next result gives a compact formula for left homotopy of Maurer--Cartan elements. In the next section, we will show that the formula specialises to gauge equivalence in the case where $V$ is a dg(l)a.

\begin{thm} \hfill
  \label{result:gauge_short_a}
  \begin{enumerate}
    \item Let $V$ be a complete $L_{\infty}$-algebra. Two Maurer--Cartan elements $\xi, \eta \in V$ are left homotopic if and only if their representing cdga maps $\xi, \eta \colon S\Sigma^{-1}V^* \to k$ satisfy
      \[
        \eta = \xi \circ e^{[\tilde{x},d]}, 
      \]
      where $\tilde{x}$ is the constant degree $-1$ derivation of $S\Sigma^{-1}V^*$ induced by the left homotopy.

    \item Let $V$ be a complete $A_{\infty}$-algebra. Two Maurer--Cartan elements $\xi, \eta \in V$ are left homotopic if and only if their representing dga maps $\xi, \eta \colon T\Sigma^{-1}V^* \to k$ satisfy
      \[
        \eta = \xi \circ e^{[\tilde{x},d]}, 
      \]
      where $\tilde{x}$ is the constant degree $-1$ derivation of $T\Sigma^{-1}V^*$ induced by the left homotopy.
  \end{enumerate}
\end{thm}

\begin{proof}
  Consider first the $L_{\infty}$-case.
  We lift the homotopy $H_{\xi,x}$ between $\xi$ and $x * \xi$ to $SU$ in the following sense: 
  Let $f$ be the cdga map $H_{\id,x} \circ e^{\theta} \circ i \colon SU \to SU$. Then $x * \xi = \xi \circ f$ and the identity map of $SU$ is left homotopic to $f$ via $H_{\id, x} \colon C(SU) \to SU$, defined by 
  \[
    H_{\id,x}(u) = u,\  
    H_{\id,x}(\bar{u}) = x(\bar{u}),\ 
    H_{\id,x}(\widehat{u}) = 0.
  \]

  We now show that $f = e^{[\tilde{x},d]}$, where $\tilde{x}$ is the constant derivation of $SU$ corresponding to $x$. 
  First we convert the homotopy $H_{\id,x}$ into a Sullivan homotopy between the identity morphism of $SU$ and $f$.  
  Consider the map
  \begin{equation}\label{eq:blhomotopy}
    e^{z\theta} + se^{z\theta} dz \colon C(SU) \to C(SU)[z,dz],
  \end{equation}
  which is well-defined as any element $u+\bar{u}+\tilde{u} \in C(SU)$ satisfies $\theta^N(u+\bar{u}+\tilde{u}) = 0$ for sufficiently large $N$ (see \cref{cyl_cdga}), so $e^{z\theta}$ is indeed a polynomial in $z$.
  By \cite[Theorem 3.4]{bl05}, equation~\eqref{eq:blhomotopy} defines a Sullivan homotopy between the identity morphism and the automorphism $e^{\theta} = e^{[s,D]}$ of $C(SU)$.
  Then defining $F, G \colon SU \to SU[z]$ to be the compositions
  \[
    F = H_{\id,x} \circ e^{z\theta} \circ i,\ G = H_{\id,x} \circ se^{z\theta} \circ i,
  \]
  we obtain that $F+Gdz \colon SU \to SU[z,dz]$ is a Sullivan homotopy from $\id$ to $f$.
  Since the constant term of $F$ is always the identity on $SU$, the map $F$ is formally invertible and the integral formula from~\cite{bl05} gives 
  \[
    f = \exp\left[\,\int_0^1 GF^{-1}\,dz, \,d \,\right].
  \]
  Finally we show that $G=\tilde{x}F$, from which it follows immediately that the integral converges and evaluates to $\tilde{x}$, concluding the proof of the theorem.
  Indeed, $H_{\id,x}s$ and $\tilde{x}H_{\id,x}$ are both $H_{\id,x}$-derivations $S(U \oplus \bar{U} \oplus \widehat{U}) \to SU$, and they agree on $U \oplus \bar{U} \oplus \widehat{U}$:
  \begin{align*}
    &H_{\id,x}s(u) = H_{\id,x}(\bar{u}) =  x(\bar{u}),\ 
    H_{\id,x}s(\bar{u}) = H_{\id,x}(\widehat{u}) = 0,\ 
    H_{\id,x}s(\widehat{u}) = 0,\\
    \shortintertext{and}
    &\tilde{x}H_{\id,x}(u) = \tilde{x}(u) = x(\bar{u}),\ 
    \tilde{x}H_{\id,x}(\bar{u}) = \tilde{x}x(\bar{u}) = 0,\ 
    \tilde{x}H_{\id,x}(\widehat{u}) = 0. 
  \end{align*}
  Hence $H_{\id,x}s=\tilde{x}H_{\id,x}$, which gives $G=\tilde{x}F$ as required.

  Now suppose $\xi$ is a Maurer--Cartan element in an $A_{\infty}$-algebra. In the $A_{\infty}$-case, the integral formula no longer applies due to the lack of graded commutativity. 
  However, we can reduce to the $L_{\infty}$-case as follows.
  Since $k$ is commutative, its representing dga map $\xi \colon TU \to k$ factors as $\xi=\xi' \circ p$, where $p \colon TU \to SU$ is the canonical projection and $\xi' \colon SU \to k$ is a cdga map. 
  Similarly there are factorizations $x * \xi = (x * \xi)' \circ p$ and $H = H' \circ p$. 
  Then $\xi'$ and $(x * \xi)'$ are Maurer--Cartan elements in the $L_{\infty}$-algebra represented by $SU$, and the cdga map $H'$ defines a left homotopy between them. 
  By definition $(x * \xi)' = x * \xi'$, hence by the $L_{\infty}$-case,
  \[
    x * \xi = (x * \xi') \circ p = \xi' \circ e^{[\tilde{x},d_{SU}]} \circ p
    = \xi' \circ p \circ e^{[\tilde{x},d_{TU}]}
    = \xi \circ e^{[\tilde{x},d_{TU}]}. 
  \]
  This proves the $A_{\infty}$-case.
\end{proof}

We would like to extend~\cref{result:gauge_short_a} to Maurer--Cartan elements in $L_{\infty}$-algebras and $A_{\infty}$-algebras of the form $V \otimes A$, that are not necessarily complete.
However, given two left-homotopic Maurer--Cartan elements $\xi, \eta \in V \otimes A$, it is \emph{not} true that their representing cdga maps $\xi, \eta \colon S\Sigma^{-1} V^* \to A$ satisfy $\eta = \xi \circ e^{[\tilde{x},d]}$ for some degree $-1$ derivation $\tilde{x}$ of $S\Sigma^{-1}V^*$, as the following counterexample shows.

\begin{exmp}
  Take $V$ to be a dg vector space, so that it has a linear differential and a decomposition $V = H(V) \oplus \Sigma B \oplus B$, and take $A=S\Sigma^{-1}V^*$. 
  Let $\xi$ be the identity map, and $\eta$ be the map induced by the projection of $V$ onto its homology $H(V)$.
  Then $\xi$ and $\eta$ are left homotopic, but $\eta$ is not an automorphism, so cannot be of the form $\id \circ e^{[\tilde{x},d]}$.
  Indeed, if $\tilde{x}$ is the constant derivation corresponding to the homotopy, then the exponential $e^{[\tilde{x},d]}$ diverges.
\end{exmp}

To obtain an analogue of~\cref{result:gauge_short_a} for Maurer--Cartan elements in $V \otimes A$ requires introducing a \emph{semi-completed symmetric algebra} and a \emph{semi-completed tensor algebra}: for a pseudocompact vector space $V$, define
\[
  S'V = \bigoplus_{i \ge 0} S^i{V} \quad \text{and} \quad
  T'V = \bigoplus_{i \ge 0} T^i{V}.
\]
Since $V$ is pseudocompact, $S^i(V)$ and $T^i(V)$ are still assumed to mean completed tensor powers. However, $S'V$ and $T'V$ differ from $\widehat{S}V$ and $\widehat{T}V$ by taking the direct sum of tensor powers instead of the direct product. 
Thus $S'V$ and $T'V$ are not pseudocompact, but do have some non-discrete topology.
They have the following property.

\begin{lem}
  \label{result:universal}
  Let $V$ be a pseudocompact vector space. 
  \begin{enumerate}
    \item Let $B$ be a pseudocompact \textup(c\textup)dga. 
      Any continuous linear map $V \to B$ extends uniquely to a continuous cdga map $S'V \to B$ or a continuous dga map $T'V \to B$.
    \item Any continuous linear map $V \to S'V$ extends uniquely to a continuous derivation of $S'V$, and any continuous linear map $V \to T'V$ extends uniquely to a continuous derivation of $T'V$.
  \end{enumerate}
\end{lem}

\begin{proof} \hfill
  \begin{enumerate}
    \item Since elements of $S'V$ and $T'V$ are finite sums of tensor powers, it suffices to prove that a linear map $f \colon V = \invlim_i V_i \to B$ extends to continuous maps $V^{\widehat{\otimes}n} \to B$ for all $n \ge 2$. Since $f$ is determined by $V_i \to B$, we define $f^{\otimes n} \colon V_{i_1} \otimes V_{i_2} \otimes \dots \otimes V_{i_n} \to B$, and take the inverse limit to obtain a map on $V^{\widehat{\otimes}n}$.
    \item This follows the same argument as above, but instead we extend $f$ to $V_{i_1} \otimes V_{i_2} \otimes \dots \otimes V_{i_n}$ by $\sum_{j=0}^{n-1} 1^j \otimes f \otimes 1^{n-1-j}$. \qedhere
  \end{enumerate}
\end{proof}

This allows us to give an alternative characterization of Maurer--Cartan elements as continuous (c)dga maps.

\begin{lem}
  \label{result:mccts}
  Let\/ $V$ be a finite-dimensional complete $L_{\infty}$-algebra, and let $A$ be a cdga. 
  There is a correspondence 
  \[
    \MC(V \otimes A) \cong \Hom (S'\Sigma^{-1}(V \otimes A)^*,k).
  \]
  Let\/ $V$ be a finite-dimensional complete $A_{\infty}$-algebra.
  There is a correspondence 
  \[
    \MC(V \otimes A) \cong \Hom (T'\Sigma^{-1}(V \otimes A)^*,k).
  \]
\end{lem}

We recover the usual representing (c)dga maps if $V$ is finite-dimensional and $A=k$.

\begin{proof}
  First note that the object $S'\Sigma^{-1}(V \otimes A)^*$ makes sense: $V$ is finite-dimensional, so $V \otimes A$ is discrete and its dual is pseudocompact. 
  Recall that $V \otimes A \cong ((V \otimes A)^*)^*$. So there is a correspondence between $(\Sigma(V \otimes A))^0$ and continuous degree 0 linear maps $\Sigma^{-1}(V \otimes A)^* \to k$, which correspond to continuous degree 0 algebra maps $S'\Sigma^{-1}(V \otimes A)^* \to k$ by \cref{result:universal}.
  The Maurer--Cartan condition corresponds to the correct axioms for differentials on $S'\Sigma^{-1}(V \otimes A)^*$ by the same argument as the usual (non-continuous) cdga case.
\end{proof}

\begin{thm} \hfill
  \label{result:gauge_short_b}
  \begin{enumerate}
    \item Let\/ $V$ be a finite-dimensional complete $L_{\infty}$-algebra and $A$ be a cdga. 
      Two Maurer--Cartan elements\/ $\xi, \eta \in V \otimes A$ are left homotopic if and only if their representing continuous cdga maps\/ $\xi', \eta' \colon S'\Sigma^{-1}(V \otimes A)^* \to k$ satisfy
      \[
        \eta' = \xi' \circ e^{[\tilde{x}',d']},
      \]
      where $\tilde{x}'$ is the constant degree $-1$ derivation of\/ $S'\Sigma^{-1}(V\otimes A)^*$ induced by the left homotopy.

    \item Let\/ $V$ be a finite-dimensional complete $A_{\infty}$-algebra and $A$ be a cdga. 
      Two Maurer--Cartan elements\/ $\xi, \eta \in V \otimes A$ are left homotopic if and only if their representing continuous dga maps\/ $\xi', \eta' \colon T'\Sigma^{-1}(V\otimes A)^* \to k$ satisfy
      \[
        \eta' = \xi' \circ e^{[\tilde{x}',d']},
      \]
      where $\tilde{x}'$ is the constant degree $-1$ derivation of\/ $T'\Sigma^{-1}(V\otimes A)^*$ induced by the left homotopy.
  \end{enumerate}
\end{thm}

\begin{proof}
  We prove the $L_{\infty}$-case; the $A_{\infty}$-case can be reduced to the $L_\infty$-case as before.
  For $U = \Sigma^{-1}(V \otimes A)^*$, consider 
  \[
    C(S'U) \coloneqq (S'(U \oplus \bar{U} \oplus \widehat{U}), D),
  \]
  where $\bar{U} \cong \Sigma U$ and $\widehat{U} \cong U$, and $D$ are defined as in $C(SU)$. 
  As before, we can also define an automorphism $e^{\theta} = \sum_{n=0}^{\infty} \theta^n / n!$ of $C(S'U)$; note that the Sullivan condition still holds on the differential $D$ of $S'U$, so the series still converges.

  While $C(S'U)$ is not a cylinder object, we can treat it as if it were one: 
  by \cref{result:universal} and \cref{result:mccts}, $\xi$ and $\eta$ are left homotopic if and only if there is a cdga map $H'$ such that the diagram commutes:
  \begin{center}
    \begin{tikzcd}
      S'U \arrow[d,shift left,"e^{\theta} \circ\, i"] \arrow[d,shift right,swap,"i"]
      \arrow[r,shift left,"\xi'"] \arrow[r,shift right,swap,"\eta'"] & k \\
      C(S'U) \arrow[ur,swap,dashed,bend right,"H'"] & 
    \end{tikzcd}
  \end{center} 
  The rest of the proof is the same as \cref{result:gauge_short_a}, replacing $C(SU)$ with $C(S'U)$ everywhere. 
  The Sullivan condition still holds on the differential $d'$ of $S'U$, so the exponential $e^{[\tilde{x}',d']}$ converges as before. 
\end{proof}

\section{Left homotopy and gauge equivalence}
\label{gauge_formulae}

In this section, \cref{result:gauge_short_a} and \cref{result:gauge_short_b} are used to obtain combinatorial formulae for the Maurer--Cartan element $x * \xi = H_{\xi,x} \circ e^{\theta} \circ i$, in terms of rooted trees. 
The formulae will show that \cref{result:gauge_short_a} and \cref{result:gauge_short_b} specialise to gauge equivalence in the case of a dgla or dga.
We will use this to deduce the following theorem at the end of the section.

\begin{thm}
  \label{thesame}
  Let\/ $V$ be a complete dgla and\/ $A$ be a cdga. For any two Maurer--Cartan elements\/ $\xi$ and $\eta$ in $V \otimes A$, the following are equivalent:
  \begin{enumerate}
    \item\label{gaugeeq} $\xi$ and $\eta$ are gauge equivalent; 
    \item\label{lefthtp} $\xi$ and $\eta$ are left homotopic;
    \item\label{sullivanhtp} $\xi$ and $\eta$ are Sullivan homotopic.
  \end{enumerate}
\end{thm}

A proof that \ref{gaugeeq} and \ref{sullivanhtp} coincide already appears in the literature; see for example~\cite{ss} and~\cite[Theorem 4.4]{cl10}.
The originality of~\cref{thesame} lies in directly establishing the equivalence of \ref{gaugeeq} and \ref{lefthtp}. 
The proof that \ref{lefthtp} and \ref{sullivanhtp} coincide is purely model category theoretic, so this theorem also provides a new proof that gauge equivalence and Sullivan homotopy coincide.

The terminology and conventions we use for trees below follows those of~\cite{gk94}.
We will allow the empty tree, that is, a tree with no vertices. 
Additionally, given a rooted tree $T$ and a natural number $k$, we say that the \emph{maximal height $k$ sub-tree of\/ $T$} is the rooted sub-tree of $T$ consisting of all vertices with a path to the root with length at most $k$, together with their internal edges and leaves.

\begin{thm}
  \label{result:gauge_formula}
  Let $V$ be a complete $L_{\infty}$-algebra, and let $x \in V^0$, $\xi \in \MC(V)$. Then
  \[
    x * \xi = \sum_T \frac{(-1)^n r}{n! j_1! \dots j_n!} T(x,\xi),
  \] 
  where the sum is taken over all rooted trees $T$ such that every vertex has at least one leaf, and for each rooted tree $T$, 
  \begin{enumerate}
    \item $n$ is the number of vertices of $T$;
    \item $r$ is the number of orderings of the vertices of $T$ such that each vertex is greater than its parent;
    \item $T(x,\xi)$ is the unique word associated to $T$ given by labelling exactly one leaf on each vertex by $x$ and all remaining leaves by $\xi$, and associating to each degree $i$ vertex with inputs $\eta_1, \dots, \eta_{i-1}, x$ the operation $[\eta_1, \dots, \eta_{i-1}, x]$;
    \item $j_1, \dots, j_n$ are the numbers of $\xi$ attached to each of the $n$ vertices.
  \end{enumerate}
\end{thm}

\begin{proof} 
  As before, we write $U = \Sigma^{-1} V^*$. 
  Using that $\tilde{x}^2 = 0$ (as the square of a constant odd derivation is 0), the same calculation as for the series $e^{\theta}$ in equation~\eqref{eq:exptheta} gives
  \[
  e^{[\tilde{x},d]} (u) = u + \sum_{n=1}^{\infty} \frac{(\tilde{x}d)^n(u)}{n!}
  \]
  for $u \in U$.
  From \cref{result:gauge_short_a} and \cref{result:gauge_short_b}, the Maurer--Cartan element $x * \xi$ is represented by the cdga map $\xi \circ e^{[\tilde{x},d]} \colon SU \to k$, so the restriction of $\xi \circ e^{[\tilde{x},d]}$ to $U$ is $\ev_{x * \xi}$. 
  Applying $\xi \circ e^{[\tilde{x},d]}$ to an element $u \colon \Sigma V \to k$ in $U$ is equivalent to forming a tree by successive compositions.
  Since $\tilde{x}$ and $d$ are derivations, each $(\tilde{x}d)^n (u)$ is a sum of words determined by sequences $d_{i_1}$, $d_{i_2}$, \dots, $d_{i_n}$, for any $i_1, \dots, i_n \ge 1$. 
  Hence at each step:
  \begin{enumerate}
    \item Applying $d_i$ to $u \in U$ gives the composition $\frac{1}{i!} u \circ \ell_i$. 
      Applying $\id^{\otimes j-1} \otimes d_i \otimes \id^{\otimes k-j}$ to an element of $U^{\otimes k}$ therefore corresponds to composition with an $i$-star along $j$. 
    \item Applying $\tilde{x}$ to $u \in U$ is the evaluation $\ev_x(u)$. 
      Applying $\id^{\otimes j-1} \otimes \tilde{x} \otimes \id^{\otimes k-j}$ to an element of $U^{\otimes k}$ therefore corresponds labelling the $j$th leaf with $x$. 
    \item Applying $\xi$ to $u \in U$ is the evaluation $\ev_{\xi}(u)$. 
      Since $\xi$ extends to a cgda map, this corresponds to labelling all remaining leaves with $\xi$.
  \end{enumerate}

  Every sequence $d_{i_1}$, $d_{i_2}$, \dots, $d_{i_n}$ gives words of the form $(-1)^n T(x,\xi) / i_1! \dots i_n! $. Indeed, regard $\xi$ and $x$ as elements in $\Sigma V$, so that $\xi$ and $x$ have degrees 0 and $-1$ respectively.
  Then by graded symmetry of the $\ell_i$, there must be exactly one $x$ on each vertex, and every term can be written as $\ell_i(\eta_1,\dots,\eta_{i-1},x)$, which equals $[\eta_1,\dots,\eta_{i-1},x]$ by our grading convention.
  Finally, each $\tilde{x}d$ introduces a sign $-1$, as both $\ell_i$ and $x$ both have odd degree.

  To determine the coefficient, it remains to count how many ways compositions give rise to the same word.
  By graded commutativity, we can form the trees such that each composition or labelling by $x$ always fills the last unlabelled leaf on each vertex.
  With this restriction, the number of ways a tree can be built is $r$, the number of monotone orderings of its vertices.
\end{proof}

When $V$ is a dgla, the above formula only allows rooted trees with vertices of valence 1 or 2, and the coefficients $r$, $j_1$, \dots, $j_n$ equal 1 for every tree. 
This recovers the formula~\eqref{eq:gauge_dgla} for gauge equivalence in dglas, and recovers the formula of~\cite[Proposition 5.9]{get09} in the case of $L_{\infty}$-algebras. 

Similarly we can use \cref{result:gauge_short_a} and \cref{result:gauge_short_b} to obtain the following analogous formula for the $A_{\infty}$-case.

\begin{thm}
  \label{result:gauge_formula_ainf}
  Let $A$ be a complete $A_{\infty}$-algebra, and let $x \in A^0$, $\xi \in \MC(A)$. Then
  \[
    x * \xi = \sum_T \sum_{\lambda} \frac{(-1)^n}{n!} T_{\lambda}(x,\xi),
  \] 
  where the sum is taken over all planar rooted trees $T$, and for each rooted tree $T$, 
  \begin{enumerate}
    \item $n$ is the number of vertices of $T$;
    \item $\lambda$ ranges over labellings of $T$ that label $n$ leaves by $x$ and all remaining leaves by $\xi$, such that for any $1 \le k \le n$, the maximal height $k$ sub-tree of\/ $T$ has $k$ leaves labelled by $x$;
    \item $T_{\lambda}(x,\xi)$ is the word given by the labelling $\lambda$ and associating to each degree $i$ vertex the operation $m_i \colon T^i \Sigma A \to \Sigma A$.
  \end{enumerate}
\end{thm}

\begin{proof}
  The calculation is similar to the $L_{\infty}$-case in \cref{result:gauge_formula}, except the lack of graded commutativity means that each $m_i$ can take more than one $x$. 
  Each sequence $d_{i_1}$, $d_{i_2}$, \dots, $d_{i_n}$ gives words of the form $(-1)^n T_{\lambda}(x,\xi)$.
\end{proof}

\begin{proof}[Proof of~\cref{thesame}]
  First note that the equivalence of (\ref{lefthtp}) and (\ref{sullivanhtp}) is immediate by completeness of $V$. 
  Also, if $V$ is finite-dimensional, then the equivalence of (\ref{gaugeeq}) and (\ref{lefthtp}) is immediate by \cref{result:gauge_short_b} and \cref{result:gauge_formula}.  
  Finally in the infinite dimensional case, by completeness of $V$ we have $V = \invlim V_i$ where $V_i$ are all finite-dimensional complete dglas.
  Hence (\ref{gaugeeq}) and (\ref{lefthtp}) are equivalent for $V$, as they are equivalent for each $V_i$.
\end{proof}

\begin{rem}
  In the case where $V \otimes A$ is complete (in particular, when $A=k$), we can simplify the above proof, as the result is just a direct consequence of \cref{result:gauge_short_a} and \cref{result:gauge_formula}.
\end{rem}

\section{A strong homotopy Poincar\'e lemma}
\label{application}

The Poincar\'e lemma states that on a contractible manifold, every closed differential form of positive degree is exact.
The following non-abelian analogue to the Poincar\'e lemma is proved in~\cite[Theorem 3.1]{vor12}.  

\begin{thm}[Non-abelian Poincar\'e lemma]
  \label{result:poincare_dgla}
  Let $M$ be a contractible manifold and let $\lie{g}$ be a dgla. 
  Let $\xi$ be a $\lie{g}$-valued differential form on $M$ such that $\xi$ is a Maurer--Cartan element in $\lie{g} \otimes \Omega(M)$. Then $\xi$ is gauge equivalent to a constant.
\end{thm}

It was suggested by Voronov that~\cref{result:poincare_dgla} may be extended to $L_{\infty}$-algebras. 
Here, we prove such a statement as an application of the results from the previous sections. 

\begin{thm}[Strong homotopy Poincar\'e lemma]
  Let $M$ be a contractible manifold and let $\lie{g}$ be a complete $L_{\infty}$-algebra. 
  If\/ $\xi$ is a $\lie{g}$-valued differential form on $M$ that is Maurer--Cartan, then $\xi$ is gauge equivalent to a constant.
\end{thm}

\begin{proof}
  By~\cref{thesame}, two Maurer--Cartan elements in $\lie{g} \otimes \Omega(M)$ are gauge equivalent if and only if their representing cdga maps are left homotopic. 
  But $M$ is contractible, so $\Omega(M)$ is weakly equivalent to $\mathbb{R}$. By nilpotence of $\lie{g}$, homotopy classes of maps $S\lie{g} \to \Omega(M)$ correspond to homotopy classes of maps $S\lie{g} \to \mathbb{R}$. Hence $\mathscr{MC}(\lie{g},\Omega(M)) \cong \mathscr{MC}(\lie{g})$.
\end{proof}


\setlength{\parindent}{0pt}
\end{document}